\documentclass[11pt, twoside, letterpaper]{amsart}

\usepackage{t1enc}
\usepackage{latexsym}
\usepackage{amssymb}
\usepackage{graphicx}
\usepackage{amsmath}
\usepackage{amsthm}
\usepackage{amsfonts}
\usepackage{mathtools}
\usepackage{comment}
\usepackage{mathrsfs}
\usepackage{enumerate}

\usepackage[all]{xy}
\usepackage[british]{babel}

\usepackage[hypertexnames=false,
    pdftex,
	pdfpagemode=UseNone,
	breaklinks=true,
	extension=pdf,
	colorlinks=true,
	linkcolor=blue,
	citecolor=blue,
	urlcolor=blue,
]{hyperref}

\newtheorem{thm}{Theorem}[section]
\newtheorem{prop}[thm]{Proposition}

\newtheorem{cor}[thm]{Corollary}

\theoremstyle{definition}
\newtheorem{exs}{Examples}[section]

\newtheorem{rmk}{Remark}[section]
\newtheorem{defn}{Definition}[section]

\numberwithin{equation}{section}

\newcommand{\C}{\mathbb{C}}

\renewcommand{\P}{\mathbb{P}}
\newcommand{\Q}{\mathbb{Q}}
\newcommand{\R}{\mathbb{R}}
\newcommand{\Z}{\mathbb{Z}}

\newcommand{\cC}{\mathcal{C}}

\newcommand{\cE}{\mathcal{E}}

\newcommand{\cO}{\mathcal{O}}
\newcommand{\cT}{\mathcal{T}}

\newcommand{\cS}{\mathcal{S}}

\newcommand{\gX}{\mathfrak{X}}
\newcommand{\gW}{\mathfrak{W}}
\newcommand{\gw}{\mathfrak{w}}

\DeclareMathOperator{\Pic}{Pic}

\DeclareMathOperator{\Hom}{Hom}

\DeclareMathOperator{\rk}{rk}

\DeclareMathOperator{\Spec}{Spec}

\DeclareMathOperator{\Coh}{Coh}

\DeclareMathOperator{\Todd}{Todd}

\DeclareMathOperator{\Sch}{Sch}

\DeclareMathOperator{\gr}{gr}

\DeclareMathOperator{\vol}{vol}

\DeclareMathOperator{\Knum}{K_\text{num}}

\DeclareMathOperator{\Amp}{Amp}
\DeclareMathOperator{\An}{An}
\DeclareMathOperator{\car}{char}
\DeclareMathOperator{\ch}{ch}

\newcommand{\leqp}{%
  \mathrel{\raisebox{-0.5ex}{$\scriptscriptstyle($}}%
  \leq
  \mathrel{\raisebox{-0.5ex}{$\scriptscriptstyle)$}}%
}

\usepackage{xcolor}

\begin{document}

\title[Slope-semistability]{Slope-semistability and moduli of coherent sheaves: a survey}
\author{Mihai Pavel,  Matei Toma}

\address{Institute of Mathematics of the Romanian Academy,
P.O. Box 1-764, 014700 Bucharest, Romania
}
\email{cpavel@imar.ro}

\address{ Universit\'e de Lorraine, CNRS, IECL, F-54000 Nancy, France
}
\email{Matei.Toma@univ-lorraine.fr}

\date{\today}
\maketitle
\begin{center}
  \emph{Dedicated to the memory of Lucian B\u adescu}
\end{center}
\begin{abstract}
We survey old and new results on the existence of moduli spaces of semistable coherent sheaves both in algebraic and in complex geometry.
\end{abstract}

\setcounter{tocdepth}{1}
\tableofcontents

\section{Introduction}

In algebraic geometry and complex geometry the classification of vector bundles is an important problem that has attracted considerable attention since the early sixties. One goal of the classification was the construction of nicely behaved moduli spaces of vector bundles. This had been achieved for line bundles, but it soon became apparent that for higher rank vector bundles a restrictive condition was necessary to obtain moduli spaces with good geometric properties. To this purpose Mumford introduced the concept of slope-semistability in \cite{mumford63} for vector bundles over algebraic curves. This was later extended to cover coherent sheaves over higher dimensional bases. The theory attracted even more interest through Donaldson's work in four-manifolds differential topology. In particular the so-called Kobayashi-Hitchin correspondence relates moduli spaces of stable vector bundles on complex projective surfaces to moduli spaces of anti-self-dual connections in gauge theory, thus providing a new perspective in the study of the subject in complex geometry.

In this paper we survey the existence and construction results of moduli spaces of semistable coherent sheaves in both algebraic and complex geometry, with a particular stress on functorial aspects. We do not delve into the vast domain studying the geometric properties and applications of these moduli spaces in enumerative geometry, classification of manifolds, hyperk\"ahler geometry, gauge theory, etc. For this there exists a rich literature, see for instance \cite{HuybrechtsLehn,mestrano16survey} and the references therein. 

\subsection*{Acknowlegements: } MP was partly supported by the PNRR grant CF 44/14.11.2022 \textit{Cohomological Hall algebras
of smooth surfaces and applications}.

\section{First properties of slope-semistable sheaves}

In this section we recall the notion of slope-semistability for coherent sheaves, and discuss its first important properties in the context of both algebraic and complex-analytic geometry. Our main references here are \cite[Chapter 1]{HuybrechtsLehn} and \cite[Chapter 5]{kobayashi2014differential}.\\

\textbf{Setup.} Throughout this paper we denote by $(X,\omega)$ a polarized $n$-dimensional space which will be either
\begin{enumerate}[(a)]
    \item[(AG)] a smooth projective variety over an algebraically closed field $k$ with an integral ample class $\omega \in \Amp^1(X)$, or
    \item[(CG)] a compact complex manifold endowed with a Hermitian metric whose K\"ahler form $\omega$ is such that $\partial\bar\partial(\omega^{n-1})=0$; such a metric is called a Gauduchon metric.
\end{enumerate}

An important special subcase of (CG) that we will frequently refer to in the sequel is when $(X,\omega)$ is
\begin{enumerate}[(a)]
    \item[(KG)] a compact K\"ahler manifold endowed with a K\"ahler class $\omega \in H^{1,1}(X,\R)$.
\end{enumerate}

Also we will denote by $\Coh(X)$ the category of coherent sheaves on $X$. It should be understood that the coherent sheaves we consider are algebraic or analytic depending on whether $(X,\omega)$ is in the case (AG) or (CG) respectively. 

For a torsion-free sheaf $E \in \Coh(X)$, we define the $\omega$-slope of $E$ by
\[
    \mu^\omega(E) \coloneqq \frac{c_1(E) \cdot \omega^{n-1}}{\rk(E)}.
\]
In the algebraic setting the intersection product $c_1(E) \cdot \omega^{n-1}$ is performed in the Chow ring $A^*(X)$. Otherwise, in the complex case,   $\omega^{n-1}$ defines a class in Aeppli cohomology which may be integrated against any class in $(1,1)$-Bott-Chern cohomology, in particular against $c_1(E)$ viewed as a class in  $H^{1,1}_{BC}(X,\R)$. When $(X,\omega)$ is K\"ahler, the intersection product $c_1(E) \cdot \omega^{n-1}$ is given by cup product in the cohomology ring $H^*(X,\R)$.

The following notion of slope-semistability was introduced by Mumford \cite{mumford63} over curves, and later extended in higher dimensions by Takemoto \cite{takemoto1972stable}. We also recall in Definition \ref{defn:GM} the Gieseker-Maruyama semistability following \cite{maruyama1976openness,gieseker77}.

\begin{defn}[Slope-semistability]
A sheaf $E \in \Coh(X)$ is \textit{$\omega$-semistable} (resp. \textit{$\omega$-stable}) if 
\begin{enumerate}
    \item $E$ is torsion-free, 
    \item for any subsheaf $F \subset E$ with $0 < \rk(F) < \rk(E)$ we have
    \[
        \mu^\omega(F) \le \mu^\omega(E) \quad (\text{resp. }<).
    \]
\end{enumerate}
\end{defn}

For the terminology, we will also say (semi)stable, for slope-(semi)stable or for $\omega$-(semi)stable when $\omega$ is clear from the context. A torsion-free sheaf $E$ is called \textit{polystable} if it is isomorphic to a direct sum of stable sheaves of the same slope. 

\begin{defn}\label{defn:simple}
A coherent sheaf $E$ on $X$ is called \textit{simple} if $\Hom(E,E) \cong k$. 
\end{defn}

It is immediately shown that stable sheaves are simple.

\begin{defn}[Gieseker-Maruyama-semistability]\label{defn:GM}
In the (AG) and (KG) setups, one defines the Hilbert polynomial of a coherent sheaf $E$ with respect to $\omega$ by setting $$P_\omega(E,m) = \int_X \ch(E)e^{m\omega}\Todd_X.$$
A sheaf $E \in \Coh(X)$ is said to be \textit{Gieseker-Maruyama (GM) semistable} (resp. \textit{Gieseker-Maruyama-stable}) if 
\begin{enumerate}
    \item $E$ is torsion-free, 
    \item for any subsheaf $F \subset E$ with $0 < \rk(F) < \rk(E)$ we have
    \[
        \frac{P_\omega(F,m)}{\rk(F)} \le \frac{P_\omega(E,m)}{\rk(E)} \quad (\text{resp. }<) \text{ for }m \gg 0.
    \]
\end{enumerate}
\end{defn}

\begin{rmk}\label{rmk:ImplicationsStability}
It is easy to see that we have the following implications for coherent sheaves in the (AG) or (KG) setups
\[
    \omega\text{-stable} \implies GM\text{-stable} \implies GM\text{-semistable} \implies \omega\text{-semistable}.
\]
\end{rmk}

\begin{exs}
\begin{enumerate}
    \item All torsion-free sheaves of rank one and in particular all line bundles are stable with respect to any polarization.
    \item A direct sum of semistable sheaves of the same slope is also semistable.
    \item The tangent bundle of the complex projective space $\P^n$ is stable. 
    \item The tangent bundle of a complex (algebraic or non-algebraic) K3 surface is stable with respect to any polarization. In positive characteristic the stability of the tangent bundle is currently unknown \cite[Ch. 9, Sect. 4]{huybrechtsK3}.
    \item All non-algebraic compact complex surfaces admit irreducible, hence stable, rank two vector bundles \cite{brinzanescuBook}. Recall that by definition a torsion-free sheaf is \textit{irreducible} if it admits no coherent subsheaf of intermediate rank.
\end{enumerate}
\end{exs}

The notion of slope-semistability fits within the broader context of algebraic stability conditions introduced by Rudakov \cite{rudakov1997stability}, and furthemore it satisfies the H\"arder-Narasimhan property. That is, for any torsion-free sheaf $E \in \Coh(X)$ there exists a unique \textit{H\"arder-Narasimhan (HN) filtration}
\[
    0 = E_0 \subset E_1 \subset \ldots \subset E_m = E
\]
such that the factors $E_i/E_{i-1}$ are $\omega$-semistable and
\[
    \mu^\omega(E_1) > \mu^\omega(E_2/E_1) > \ldots > \mu^\omega(E/E_{m-1}).
\]
We refer the reader to \cite[Section 1.3]{HuybrechtsLehn} for a proof in the algebraic case and to \cite[Theorem 5.7.15]{kobayashi2014differential} in the analytic case.


In general one can ``approximate'' any semistable sheaf $E \in \Coh(X)$ by stable sheaves using \textit{Seshadri filtrations}
\[
   E^\bullet: \quad 0 = E_0 \subset E_1 \subset \ldots \subset E_m = E
\]
with stable factors $E_i/E_{i-1}$ of slope $\mu^\omega(E_i/E_{i-1})=\mu^\omega(E)$. We shall denote by $\gr^{S}(E^\bullet) = \bigoplus_i E_i/E_{i-1}$ the corresponding graded sheaf of such a filtration.

\begin{rmk}
We note that a semistable sheaf $E$ might admit many Seshadri filtrations, however one can show that the graded module corresponding to any such filtration is uniquely determined in codimension one. In other words, the reflexive hull $\gr^{S}(E^\bullet)^{\vee \vee}$ of the graded sheaf does not depend on the choice of the Seshadri filtration \cite[Section 1]{HuybrechtsLehn}. The graded modules of different Seshadri filtrations may however be distinct, see \cite[Example 3.1]{BTT2017}.
\end{rmk}

\section{Some key results on slope-semistable sheaves}

\subsection{Set-theoretical Kobayashi-Hitchin correspondence}

We present here the Kobayashi-Hitchin correspondence which establishes a link between the algebraic geometric concept of stability and the existence of Hermite-Einstein metrics in complex differential geometry. It allows the use of analytic and differential geometric methods in the study of semistable vector bundles and their moduli spaces in complex geometry. See \cite{lubke95kobayashi} for a thorough treatment of this subject.

This works in the complex geometrical setup of compact complex manifolds endowed with a Gauduchon metric. 
Let us fix a compact Gauduchon manifold $(X,\omega)$ of dimension $n$. Let $(E,h)$ be a $\mathscr{C}^\infty$ complex vector bundle on $X$, endowed with a Hermitian metric $h$. Then, by definition, an $h$-unitary connection $A$ on $(E,h)$ is called $\omega$-\textit{Hermite--Einstein} if $A$ is integrable and satisfies
\begin{align*}
  \Lambda_\omega F_A  = - \frac{2\pi i}{(n-1)!\vol_\omega(X)} \mu^\omega(E) \text{Id}_E.
\end{align*}
Here $\Lambda_\omega$ is the adjoint of the Lefschetz operator on forms given by wedging with $\omega$.
Moreover, $A$ is called \textit{irreducible} if it has no decomposition $A = A_1 \oplus A_2$ coming from an orthogonal splitting $E = E_1 \oplus E_2$ of the Hermitian smooth vector bundle $(E,h)$. Note that the integrability condition on $A$ endows $E$ with a holomorphic structure $\cE_A = (E,\overline{\partial}_A)$ by the Newlander-Nirenberg theorem. 

\begin{thm}\label{thm:KH}
Let $(E,h_0)$ be a Hermitian complex vector bundle on the Gauduchon manifold $(X,\omega)$. If there exists an irreducible $\omega$-Hermite-Einstein connection on $(E,h_0)$, then the induced holomorphic structure $\cE_A$ on $E$ is $\omega$-stable. Conversely, if $\cE$ is an $\omega$-stable holomorphic structure on $E$, then there exists a Hermitian metric $h$ on $E$ such that its Chern connection with respect to $\cE$ is irreducible $\omega$-Hermite-Einstein on $(E,h)$. This metric $h$ is called $\omega$-Hermite-Einstein on $\cE$ and is unique up to multiplication by a positive factor. 
\end{thm}

\begin{rmk}
In the setup of Theorem \ref{thm:KH}, if $\cE$ is an $\omega$-stable holomorphic structure on $E$, then there exists an $\omega$-Hermite-Einstein connection $A$ on $(E,h_0)$ such that the induced holomorphic structure $\cE_A$ on $E$ is isomorphic to $\cE$, see \cite{lubke95kobayashi}.   
\end{rmk}

\subsection{Bogomolov inequality}
Over a smooth projective variety or compact K\"ahler manifold, the Bogomolov inequality expresses a strong topological constraint to which semistable torsion-free sheaves are subject. It is particularly useful in boundedness questions, see Theorem \ref{thm:BoundednessAG}.

We will state the Bogomolov inequality in the framework of algebraic and K\"ahler geometry, and then make a remark on its formulation in the Gauduchon setup.

\begin{thm}[Bogomolov inequality] \label{thm:Bogomolov}
In the zero characteristic (AG) setup and in the (KG) setup, for any $\omega$-semistable sheaf $E$ on $X$ one has
$$\Delta(E)\cdot \omega^{n-2} \ge 0,$$
where $$\Delta(E) := 2 \rk(E) c_2(E) - (\rk(E) - 1)c_1(E)^2.$$
In the (AG) setup in characteristic $p > 0$ and for $E$ $\omega$-semistable, we have
$$\Delta(E)\cdot \omega^{n-2} + \frac{\rk(E)^2(\rk(E)-1)^2}{(p-1)^2} \omega^n \ge 0.$$
\end{thm}

\begin{rmk}
In the algebraic case, the theorem was first proved by Bogomolov \cite{Bogomolov} over algebraic surfaces in zero characteristic. The general algebraic case in zero characteristic follows from his result and the Mehta-Ramanathan restriction theorem \ref{thm:MRrestrictions}. The positive characteristic case was proved by Langer \cite{langer2004semistable}.
\end{rmk}

\begin{rmk}
    The inequality $\Delta(E)\cdot \omega^{n-2} \ge 0$ was proved in the K\"ahler case by L\"ubke for holomorphic vector bundles $E$ admitting an $\omega$-Hermite-Einstein metric, see \cite{kobayashi2014differential,lubke95kobayashi}. Together with the Kobayashi-Hitchin correspondence it yields the statement of Theorem \ref{thm:Bogomolov} in the K\"ahler case. In fact, L\"ubke's proof also applies in the context of a Gauduchon manifold $(X,\omega)$, leading to a pointwise inequality of $(n,n)$-forms:
    \[
        \Delta(E,h)\wedge \omega^{n-2} \ge 0,
    \]
    where $h$ is an $\omega$-Hermite-Einstein metric on $E$ and the $(2,2)$-form $\Delta(E,h)$ is computed using the associated Chern connection. 
\end{rmk}

\subsection{Restriction theorems}

In this subsection we place ourselves in the algebraic setting and present restriction results for (semi)stable sheaves. These are used in moduli theory to prove boundedness and general properties of moduli spaces of sheaves. 

Let $H$ be an ample divisor representing the polarization $\omega$. We assume the dimension $n$ of $X$ to be larger than one. Let $E \in \Coh(X)$ be a torsion-free sheaf. If $D \in |aH|$ is a smooth divisor for some $a > 0$ such that $E|_D$ is (semi)stable with respect to $H|_D$, then it is immediate to see that $E$ is also (semi)stable with respect to $H$. One may wonder if a converse statement holds. The example of the tangent bundle $\cT_{\P_\C^n}$ shows that some caution is required. Indeed, its restriction to any hyperplane $D$ is isomorphic to $\cO_{\P_\C^{n-1}}(1) \oplus \cT_{\P_\C^{n-1}}$, which is not semistable. However, a positive answer is found if one takes a general divisor $D \in |aH|$ for $a$ sufficiently large. This is the content of the Mehta-Ramanathan restriction theorem \cite{mehta1982semistable,mehta1984restriction}:

\begin{thm}\label{thm:MRrestrictions}
If $E$ is a $H$-(semi)stable sheaf on $X$, then its restriction $E|_D$ to a general divisor $D \in |aH|$ of sufficiently large degree is $(H|_D)$-(semi)stable.   
\end{thm}

More refined restriction theorems which give effective bounds on the degree $a$ were proved by Flenner \cite{flenner1984restrictions} in zero characteristic and by Langer \cite[Theorem 5.2 and Corollary 5.4]{langer2004semistable} in mixed characteristic. See also the recent paper \cite{Feyzbakhsh} containing effective restriction results. We state here a variant of Langer's result.

\begin{thm}\label{thm:Langer}
If $E$ is a $H$-(semi)stable sheaf on $X$, then its restriction $E|_D$ to a general divisor $D \in |aH|$ is $(H|_D)$-(semi)stable provided that
\[
    a >  \frac{\rk(E)-1}{\rk(E)} \Delta(E)H^{n-2} + \frac{1}{\rk(E)(\rk(E)-1)H^n} + \frac{(\rk(E)-1)H^n}{\rk(E)}\gamma_{\rk(E)},
\]
where $\gamma_{r} := 0$ if $\car(k) = 0$ and $\gamma_{r} := \frac{r^2(r-1)^2}{(p-1)^2}$ if $\car(k)=p > 0$.
\end{thm}

\begin{rmk}
In Theorem \ref{thm:Langer}, ``general'' can be made explicit depending on $E$. More precisely, if $0 = E_0 \subset \ldots \subset E_m = E$ is a Seshadri filtration of $E$, then the statement holds for any smooth (even normal) divisor $D \in |aH|$ so that any factor $E_i/E_{i-1}$ restricted to $D$ remains torsion-free.
 \end{rmk}

\section{Families of slope-semistable sheaves}

In this section we present properties of families of slope-semistable sheaves which are essential in moduli theory. An $S$-flat family of coherent sheaves on $X$ is by definition a coherent $\cO_{S \times X}$-module $\cE$, flat over $S$. The parameterizing space $S$ is either a $k$-scheme or a complex analytic space, depending on whether we work in the algebraic or analytic setup respectively. 

\subsection{Boundedness of sets of coherent sheaves}

We recall below what we mean by a \textit{bounded} set of coherent sheaves on $X$. 

\begin{defn}
Let $\cS$ be a set of isomorphism classes of coherent sheaves on $X$. We say that $\cS$ is \textit{bounded} if 
\begin{enumerate}
    \item[(AG)] there exists a scheme $S$ of finite type over $k$ and  a coherent sheaf $E$ on $S \times X$ such that $\cS$ is contained in the set of isomorphism classes of fibers of $E$ over points of $S$ \cite[Definition 1.7.5]{HuybrechtsLehn}.
    \item[(CG)] there exists a complex analytic space $S$, a compact subset $K \subset S$ and a coherent sheaf $E$ on $S \times X$ such that $\cS$ is contained in the set of isomorphism classes of fibers of $E$ over points of $K$ \cite[Definition 5.1]{toma2021boundedness}.
\end{enumerate}
\end{defn}

When $X$ is complex projective, the above two definitions are in fact equivalent via the GAGA Theorem, cf. \cite[Remark 3.3]{toma2016boundednessK}. 

Note that if $\cS$ is a bounded family of sheaves on $X$, then the Chern classes (seen in the numerical group of $X$ in the algebraic case, and in the singular cohomology ring $H^*(X,\Z)$ in the analytic case respectively) of the elements in $\cS$ range within a finite set.

The following boundedness criterion is due to Grothendieck \cite{grothendieck1961techniques} in the algebraic case. The analytic version can be found in \cite{toma2021boundedness}.

 \begin{prop}\label{prop:boundednessCriterion}
Let $\cS$ be a set of isomorphism classes of torsion-free sheaves on $X$. Then $\cS$ is bounded if and only if the following two conditions are fulfilled
\begin{enumerate}
    \item $\cS$ is dominated, i.e., there exists a bounded set $\cT$ of classes of coherent sheaves on $X$ such that each element of $\cS$ is a quotient of an element of $\cT$,
    \item the slope function $\mu^\omega$ is upper bounded on $\cS$.
\end{enumerate}
 \end{prop}

\subsection{Openness of semistability}

The following result shows that slope-semistability, resp. slope-stability, is an \textit{open} property in flat families of sheaves. Its proof is based on the boundedness criterion stated in Proposition \ref{prop:boundednessCriterion}.
\begin{prop}\label{prop:Openness}
Let $(X,\omega)$ be a polarized space as in our (AG) or (KG) setups. Let $\cE$ be an $S$-flat family of coherent sheaves on $X$. Then the locus $S^\circ$ of closed points $s \in S$ such that $\cE|_{\{s\} \times X}$ is $\omega$-semistable (resp. $\omega$-stable) is a Zariski open subset of $S$. 
\end{prop}
\begin{proof}
See \cite[Proposition 2.3.1]{HuybrechtsLehn} for a proof in the algebraic case, and \cite[Corollary 6.7]{toma2021boundedness} for the analytic case. 
\end{proof}

\begin{rmk}
The above result is in general false for non-K\"ahler Gauduchon manifolds, see \cite{TelemanOpenness}. 
\end{rmk}

\subsection{Langton's semistable reduction}

We treat the algebraic case. For the analytic case see \cite{toma2020criteria}.

\begin{prop}[Semistable reduction, \cite{langton1975valuative}]\label{prop:SReduction}
Let $(X,\omega)$ be as in the (AG) setup. Let $R$ be a DVR over $k$ of quotient field $K$. Let $E$ be an $R$-flat family of coherent sheaves on $X$ such that $E_K$ is $\omega$-semistable. Then there exists a subsheaf $F \subset E$ such that $F_K \cong E_K$ and such that $F_k$ is $\omega$-semistable.
\end{prop}

\subsection{Boundedness of semistability}

The Chern character of a coherent sheaf $E$ on $X$ will determine its numerical type $\ch(E)$ in the numerical group $\Knum(X)_\Q$ in the (AG) case, respectively in the singular cohomology group $H^*(X,\Q)$ in the (CG) case. Given a class $\gamma$ in $\Knum(X)_\Q$, respectively in $H^*(X,\Q)$, we consider the following \textit{boundedness} statement:
\\

$\textbf{B}_\gamma(\omega):$ The set of isomorphism classes of coherent sheaves $E$ of class $\gamma$ on $X$ that are $\omega$-semistable is bounded. 
\\

Boundedness of semistable sheaves was intensively studied and it took the efforts of many mathematicians to completely solve the algebraic case (e.g. \cite{atiyah57vector,KleimanSga6, takemoto1972stable, gieseker77, Maruyama81boundedness,simpson1994moduli,langer2004semistable}).

\begin{thm}\label{thm:BoundednessAG}
In the (AG) setup, boundedness of semistability $\textbf{B}_\gamma(\omega)$ holds for any numerical class $\gamma$.   
\end{thm}

In the (KG) case, the boundedness statement is not yet known in full generality. We present here a few results that indicate its validity. 

\begin{thm}\label{thm:BoundednessKGI}
Let $(X,\omega)$ be a compact K\"ahler manifold and $\gamma$ be a topological class in $H^*(X,\Q)$. Then $\textbf{B}_\gamma(\omega)$ is known in the following cases:
\begin{enumerate}
    \item when $\gamma$ is the class of a rank 1 coherent sheaf \cite[Corollary 5.5]{toma2016boundednessK},
    \item when $X$ is complex projective, but $\omega$ is not necessarily an ample class \cite[Proposition 6.3]{greb2017compact}.
    \item when $X$ is a (not necessarily algebraic) K3 surface or a 2-dimensional torus and the class $\omega$ is $\gamma$-generic, see \cite[Theorem 4.3]{Toma-Documenta} and \cite[Section 2.2.1]{PeregoToma} for details.
\end{enumerate}
\end{thm}

\begin{rmk}
In the non-K\"ahler case, boundedness of semistability cannot be expected in the above formulation as can be seen from the following example. Let $X$ be a class $VII$ surface. Then $b_1(X) = 1$ and the identity component $\Pic^0(X)$ of the Picard group pf $X$ is isomorphic to $\C^*$. In this case holomorphic line bundles with $c_1 = 0$ form an unbounded family of semistable sheaves on $X$, with respect to any polarization. 

A possible remedy to this situation would be to fix the first Chern class of the line bundles in the Bott-Chern cohomology of $X$, and not only in the singular cohomology. The effect would be to fix the degree of the considered line bundles. The set parametrizing line bundles of fixed degree in the above example is a circle in $\C^*$ centered at $0$ \cite[Section 1.3]{lubke95kobayashi}.
\end{rmk}

\section{Moduli functors and moduli stacks}

\subsection{Moduli functors of sheaves}

In this subsection we let $\cC$ be the category $(\Sch \!/k)$ of $k$-schemes in the algebraic geometric setting, and the category $(\An \!/\C)$ of (not necessarily separated) complex analytic spaces in the complex geometric setting.

\begin{defn}
Let $F: \cC \to (Sets)$ be a contravariant functor and $\phi: F \to \Hom(-,M)$ a natural transformation of functors where $M$ is an algebraic space over $k$, respectively an analytic space. We say that $\phi$ is
\begin{itemize}
    \item\textit{a categorical moduli space for $F$} if any other natural transformation $\psi: F \to \Hom(-,N)$ with $N$ an algebraic space over $k$, respectively an analytic space factorizes through $\phi$. One also says that $M$ \textit{corepresents} the functor $F$ in this case.
    \item \textit{a coarse moduli space for $F$} if it is a categorical moduli space and moreover induces a bijection at the level of $k$-points $F(\Spec k) \to \Hom(\Spec k,M)$, respectively of $\C$-points.
    \item \textit{a fine moduli space for $F$} if $\phi$ is an isomorphism of functors. 
\end{itemize}
\end{defn}

Let $(X,\omega)$ be a polarized space as in the algebraic, respectively K\"ahler setup. Given a class $\gamma$ in $\Knum(X)_\Q$, respectively in $H^*(X,\Q)$, let $\Coh_{X,\gamma} : \cC \to (Sets)$ denote the functor of coherent sheaves of class $\gamma$, which sends an object $S \in \cC$ to the set of isomorphism classes of flat families of coherent sheaves of class $\gamma$ on $X$ parameterized by $S$. In the sequel we will consider the following subfunctors of the functor $\Coh_{X,\gamma}$:
\begin{itemize}
    \item $\Coh^{ss}_{X,\omega,\gamma}$ of $\omega$-semistable sheaves,
    \item $\Coh^{s}_{X,\omega,\gamma}$ of $\omega$-stable sheaves,
    \item $\Coh^{lf,s}_{X,\omega,\gamma}$ of $\omega$-stable locally free sheaves,
    \item $\Coh^{spl}_{X,\gamma}$ of simple torsion-free sheaves \cite{AltmanKleiman,KosarewOkonek},
    \item $\Coh^{(SLF)}_{X,\omega,\gamma}$ of torsion-free sheaves with Seshadri locally free graduations \cite{PavelToma-reflexive,BuchdahlSchumacher22},
    \item $\Coh^{(SR)}_{X,\omega,\gamma}$ of torsion-free sheaves with Seshadri reflexive graduations \cite{PavelToma-reflexive}.
\end{itemize}
All these functors are Zariski-open subfunctors of $\Coh_{X,\gamma}$, see Proposition \ref{prop:Openness} and the above cited references. 

We will discuss the question whether these functors admit a categorical/coarse/fine moduli space after saying a few words about the corresponding moduli stacks.

\subsection{Moduli stacks of sheaves}

Since the theory of analytic stacks is less developed, in this subsection we place ourselves in the algebraic setup and recall some known facts about the stack of coherent sheaves. For a detailed account on stacks and algebraic stacks we refer the reader to \cite{LaumonChampsAlg,OlssonStacks,gomez01algebraic,AlperNotes,stacks-project}.

Consider the category $\cC oh_X$ whose objects are pairs $(S,E)$, where $S$ is a scheme over $k$ and $E$ is an $S$-flat family of sheaves on $X$. A morphism $(S',E') \to (S,E)$ in $\cC oh_X$ consists of a map $f: S' \to S$ of $k$-schemes together with a morphism $E \to f_* E'$ of sheaves whose adjoint is an isomorphism. We visualize this as a cartesian diagram
\[
    \xymatrix{ E' \ar[r] \ar[d] &  E \ar[d]\\  S \ar[r]^f &  S }
\]
This defines a category fibered in groupoids over the category of schemes over $k$.
 
\begin{prop}[{\cite[Tag 09DS, Tag 0DLY]{stacks-project}}] \label{prop:stackCoh}
The category $\cC oh_X$ is an algebraic stack locally of finite type and with affine diagonal over $k$.    
\end{prop}

Given a numerical class $\gamma \in \Knum(X)$, we consider the open substacks $$\cC oh_{X,\omega,\gamma}^{ss}, \cC oh_{X,\omega,\gamma}^{s},\cC oh_{X,\omega,\gamma}^{lf,s},\cC oh_{X,\gamma}^{spl},\cC oh_{X,\omega,\gamma}^{(SLF)},\cC oh_{X,\omega,\gamma}^{(SR)}$$ 
of $\cC oh_X$ corresponding to the functors
$$\Coh_{X,\omega,\gamma}^{ss}, \Coh_{X,\omega,\gamma}^{s},\Coh_{X,\omega,\gamma}^{lf,s},\Coh_{X,\gamma}^{spl},\Coh_{X,\omega,\gamma}^{(SLF)},\Coh_{X,\omega,\gamma}^{(SR)}.$$ 
Theorem \ref{thm:BoundednessAG} implies that $\cC oh_{X,\omega,\gamma}^{ss}$ is quasi-compact, therefore so are $\Coh_{X,\omega,\gamma}^{s},\\ \Coh_{X,\omega,\gamma}^{lf,s},\Coh_{X,\omega,\gamma}^{(SLF)},\Coh_{X,\omega,\gamma}^{(SR)}$ too. Moreover Proposition \ref{prop:SReduction} yields that $\cC oh_{X,\omega,\gamma}^{ss}$ is universally closed over $k$. All together we obtain:

\begin{prop}
The substack $\cC oh_{X,\omega,\gamma}^{ss} \subset \cC oh_X$ is open and a universally closed algebraic stack of finite type and with affine diagonal over $k$.
\end{prop}


As in the case of moduli functors, one can define the notions of categorical/coarse/fine moduli spaces for algebraic stacks in the following way.
\begin{defn}
Let $\gX$ be an algebraic stack over $k$ and $\phi: \gX \to \Hom(-,M)$ a morphism of stacks, where $M$ is an algebraic space over $k$. We say that $\phi$ is
\begin{itemize}
    \item\textit{a categorical moduli space for $\gX$} if any other morphism $\psi: \gX \to \Hom(-,N)$ with $N$ an algebraic space over $k$ factorizes through $\phi$.
    \item \textit{a coarse moduli space for $\gX$} if it is a categorical moduli space and moreover induces a bijection between the set of isomorphism classes of $k$-points of $\gX$ and $\Hom(\Spec k,M)$.
    \item \textit{a fine moduli space for $\gX$} if $\phi$ is an isomorphism of stacks. 
\end{itemize}
\end{defn}

Note that the above moduli stacks never admit fine moduli spaces, since the automorphism groups of the objects are non-trivial. As to the existence of categorical or coarse moduli spaces for moduli stacks, this is equivalent to the existence of categorical or coarse moduli spaces for the corresponding moduli functors described above.

A more refined version of a categorical moduli space is the following.

\begin{defn}[\cite{alper2013good}]
A quasi-compact and quasi-separated morphism $\phi: \gX \to M$ from an algebraic stack $\gX$ to an algebraic space $M$ is said to be \textit{a good moduli space} if
\begin{itemize}
    \item the pushfoward functor on quasi-coherent sheaves is exact, and
    \item the induced morphism of sheaves $\cO_M \to \phi_* \cO_\gX$ is an isomorphism.
\end{itemize}
\end{defn}

Good moduli spaces are always categorical \cite[Theorem 6.6]{alper2013good}, but not coarse in general.

A natural question to be discussed next is that of the existence of a categorical/good/coarse moduli space for the above moduli stacks. 

\section{Moduli spaces of sheaves}

In algebraic and in complex geometry, the first moduli spaces of sheaves of particular interest were moduli spaces of line bundles and more generally of vector bundles. In the latter case it was soon observed that in order to obtain moduli spaces with good geometric properties (such as local-separatedness) one has to impose some restriction on the class of vector bundles to be classified. This led Mumford to introduce the slope-stability condition in \cite{mumford63}.

\begin{thm}\label{thm:stableCase}
The functor $\Coh_{X,\omega,\gamma}^{lf,s}$ admits a separated coarse moduli space $M_{X,\omega,\gamma}^{lf,s}$. In the (AG) case, $M_{X,\omega,\gamma}^{lf,s}$ is a quasi-projective scheme over $k$.
\end{thm}

This result was proved in the algebraic geometrical context using Geometric Invariant Theory methods over curves by Mumford \cite{mumford63}, Seshadri \cite{Seshadri65} (see also \cite{NarasimhanSeshadri64}), over surfaces by Gieseker \cite{gieseker77} and in higher dimensions by Maruyama \cite{maruyama77}. In the analytic setup, the result is a consequence of the existence of a coarse moduli space of simple vector bundles, proved by Norton \cite{Norton79Simple} using Banach-analytic techniques, together with the openness of stability, Proposition \ref{prop:Openness}.

\begin{rmk}
In general the moduli spaces $M_{X,\omega,\gamma}^{lf,s}$ are rarely fine \cite{maruyama78moduli}. A situation when this is known to happen is when $X$ is a curve, $\omega$ is the fundamental class of $X$, and the rank and degree of the concerned vector bundles are coprime. 
\end{rmk}

In complex geometry, moduli spaces of stable vector bundles are related via the Kobayashi-Hitchin correspondence to moduli spaces of Hermite-Einstein connections.

\begin{thm}[Moduli-theoretical Kobayashi-Hitchin correspondence, \cite{Donaldson87,kobayashi2014differential,LubkeOkonek,Miyajima89,lubke95kobayashi}]\label{thm:KHmoduli}
Let $(X,\omega)$ be a Gauduchon compact complex manifold. Let $E$ be a smooth complex vector bundle on $X$ and $h$ a Hermitian metric on $E$. Then the set-theoretical Kobayashi-Hitchin correspondence yields a real-analytic isomorphism
\[
    M^{HE}_{X,\omega,E,h} \to M^{lf,s}_{X,\omega,E}
\]
between the moduli space of $\omega$-Hermite-Einstein connections on $(E,h)$ and the moduli space of $\omega$-stable holomorphic structures on $E$.
\end{thm}

In general the moduli spaces $M_{X,\omega,\gamma}^{lf,s}$ are not compact. If one aims at constructing natural compactifications, one generally needs to relax both the locally-freeness and the stability condition of the coherent sheaves to be parametrized. 

We have already mentioned moduli spaces of simple vector bundles. These extend to (not necessarily separated) moduli spaces of simple torsion-free sheaves.

\begin{thm}[\cite{AltmanKleiman,KosarewOkonek}]
The functor $\Coh_{X,\gamma}^{spl}$ admits a coarse moduli space $M_{X,\gamma}^{spl}$.
\end{thm}

\begin{cor}
 The functor $\Coh_{X,\omega,\gamma}^{s}$ admits a separated coarse moduli space $M_{X,\omega,\gamma}^{s}$ as an open subset in $M_{X,\gamma}^{spl}$.  
\end{cor}

We note that in the algebraic setup one can further show that the moduli space $M_{X,\omega,\gamma}^{s}$ is quasi-projective over $k$ \cite{maruyama77}.

When $X$ is a curve, the functor $\Coh_{X,\omega,\gamma}^{ss}$ admits a projective categorical moduli space $M_{X,\omega,\gamma}^{ss}$, which contains $M_{X,\omega,\gamma}^{s}$ as an open subscheme \cite{Seshadri65}. The geometric points of $M_{X,\omega,\gamma}^{ss}$ correspond to isomorphism classes of Seshadri graduations of semistable sheaves, and therefore $M_{X,\omega,\gamma}^{ss}$ is not a coarse moduli space in general.

When trying to employ Geometric Invariant Theory to construct compactifications of $M_{X,\omega,\gamma}^{s}$ in higher dimensions, one is led to consider the Gieseker-Maruyama-semistability condition, see Definition \ref{defn:GM}.

\begin{thm}[\cite{gieseker77,maruyama77,simpson1994moduli,alper2013good}]
In the algebraic setup and in zero characteristic, the substack $\cC oh_{X,\omega,\gamma}^{GMss} \subset \cC oh_{X,\gamma}$ of Gieseker-Maruyama-semistable sheaves is open and admits a projective good moduli space $M_{X,\omega,\gamma}^{GMss}$.  
\end{thm}

\begin{rmk}
In positive characteristic, the above statement holds if one replaces ``good moduli'' by ``adequate moduli'', see Alper \cite{AlperAdequate}.
\end{rmk}
\begin{rmk}
In the general (KG) setup, the existence of the Gieseker-Maruyama moduli space is generally unknown. There exist partial results when $X$ is complex projective and $\omega$ is a non-ample K\"ahler class \cite{greb2020moduliKahler}.
\end{rmk}

A different way to enlarge the open substack $\cC oh_{X,\omega,\gamma}^{lf,s}$ is to look at $\cC oh_{X,\omega,\gamma}^{(SLF)}$ and $\cC oh_{X,\omega,\gamma}^{(SR)}$. For these one still gets good moduli spaces in the (AG) setting and characteristic zero \cite{PavelToma-reflexive}. See also \cite{BuchdahlSchumacher22} for the complex analytic case.

\section{Further topics}\label{sect:further}

\subsection{Donaldson-Uhlenbeck compactification}

In this subsection we will consider the case when $(X,\omega)$ is a polarized smooth complex projective surface. Let $E$ be a smooth complex vector bundle on $X$ and $h$ a Hermitian metric on $E$. Recall that by the moduli-theoretical Kobayashi-Hitchin correspondence there is a real-analytic isomorphism
\[
    M^{HE}_{X,\omega,E,h} \to M^{lf,s}_{X,\omega,E}
\]
between the moduli space of $\omega$-Hermite-Einstein connections on $(E,h)$ and the moduli space of $\omega$-stable holomorphic structures on $E$. 

It is important in Donaldson theory to work with suitable compactifications of moduli spaces of anti-self-dual connections. These were constructed by Donaldson based on compactness results due to Uhlenbeck \cite{donaldson90geometry} and lead also to a compactification $M^{DU}_{{X,\omega,E,h}}$ of $M^{HE}_{X,\omega,E,h}$, which we call the Donaldson-Uhlenbeck compactification.

On the algebraic geometrical side, we have already seen a compactification of the moduli space $M^{lf,s}_{X,\omega,E}$ of slope-stable vector bundles by adding Gieseker-Maruyama-semistable torsion-free sheaves at the boundary, which is the Gieseker-Maruyama moduli space $M^{GMss}_{X,\omega,\ch(E)}$. Le Potier \cite{le1992fibre} and Jun Li \cite{li1993algebraic} constructed a projective morphism
\[
    \varphi: M^{GMss}_{X,\omega,\ch(E)} \to \P_\C^N
\]
which is an immersion on $M^{lf,s}_{X,\omega,E}$.
Furthermore, Li proved that the closure of the image $\varphi(M^{lf,s}_{X,\omega,E})$ inside $\P_\C^N$ is homeomorphic to the Donaldson-Uhlenbeck compactification $M^{DU}_{{X,\omega,E,h}}$. This extends the inverse of the Kobayashi-Hitchin correspondence as a homeomorphism of compact spaces 
\[
    \overline{\varphi(M^{lf,s}_{X,\omega,E})} \to M^{DU}_{{X,\omega,E,h}}.
\]
In particular one can transfer the complex algebraic structure of $\overline{\varphi(M^{lf,s}_{X,\omega,E})}$  to the Donaldson-Uhlenbeck compactification.

As a further compactification of $M^{lf,s}_{X,\omega,\gamma}$, Huybrechts and Lehn constructed in \cite[Chapter 8]{HuybrechtsLehn} a complex projective moduli space $M^{\mu ss}_{X,\omega,\gamma}$ of slope-semistable sheaves over a smooth surface, which comes together with a natural transformation of functors $$\Coh_{X,\omega,\gamma}^{ss} \to \Hom(-,M^{\mu ss}_{X,\omega,\gamma}).$$
However, $M^{\mu ss}_{X,\omega,\gamma}$ does not corepresent the moduli functor in general. 

Similar results were obtained by Greb, Sibley, Toma, Wentworth \cite{greb2021HYM} in dimension larger than two using the analogue of the Donaldson-Uhlenbeck compactification due to Tian \cite{tian2000gauge} and the higher dimensional analogue of the Huybrechts-Lehn moduli space $M^{\mu ss}_{X,\omega,\gamma}$.

\subsection{Moduli of pure sheaves}

Until now we have only considered classification problems of torsion-free sheaves. It is however natural to extend this research to lower-dimensional coherent sheaves. The analogue of the torsion-free condition in this situation is the \textit{purity} condition.

\begin{defn}
A coherent sheaf $E$ on $X$ is said to be \textit{pure of dimension $d$} if any non-trivial coherent subsheaf $F \subset E$ has dimension $d$ too.
\end{defn}

One defines the notions of $\omega$-(semi)stability and $GM$-(semi)stability also for pure sheaves in the (AG) and (KG) setups. For this one writes the Hilbert polynomial of a coherent sheaf $E$ of dimension $d$ on $X$ in the following form
\[
    P_\omega(E,m) = \sum_{i=0}^d \alpha_i(E) m^i,\quad   \text{where }\alpha_i(E) = \frac{1}{i!}\int_X \ch(E)\omega^i \Todd_X.
\]

\begin{defn}
A coherent sheaf $E$ of dimension $d$ on $(X,\omega)$ is said to be $\omega$\textit{-(semi)stable} if
\begin{enumerate}
    \item $E$ is pure, 
    \item for any coherent subsheaf $F \subset E$ with $0 < \alpha_d(F) < \alpha_d(E)$ we have
    \[
        \frac{\alpha_{d-1}(F)}{\alpha_d(F)} \leqp \frac{\alpha_{d-1}(E)}{\alpha_d(E)}.
    \]
\end{enumerate}
\end{defn}
\begin{defn}
A coherent sheaf $E$ of dimension $d$ on $(X,\omega)$ is said to be $GM$\textit{-(semi)stable} if
\begin{enumerate}
    \item $E$ is pure, 
    \item for any coherent subsheaf $F \subset E$ with $0 < \alpha_d(F) < \alpha_d(E)$ we have
    \[
        \frac{P_\omega(F,m)}{\alpha_d(F)} \leqp \frac{P_\omega(E,m)}{\alpha_d(E)} \quad \text{for }m\gg 0.
    \]
\end{enumerate}    
\end{defn}

One can also consider the following notion of semistability that interpolates between slope-semistability and GM-semistability.
\begin{defn}
For integers $1 \le \ell \le d \le n$, a coherent sheaf $E$ of dimension $d$ on $(X,\omega)$ is said to be $\ell$\textit{-(semi)stable} if
\begin{enumerate}
    \item $E$ is pure, 
    \item for any coherent subsheaf $F \subset E$ with $0 < \alpha_d(F) < \alpha_d(E)$ we have
    \[
        \frac{\sum_{i=d-\ell}^d \alpha_i(F) m^i}{\alpha_d(F)} \leqp \frac{\sum_{i=d-\ell}^d \alpha_i(E) m^i}{\alpha_d(E)} \quad \text{for }m\gg 0.
    \]
\end{enumerate}    
\end{defn}

For any class $\gamma$ and $\ell$ between $1$ and $\dim(\gamma)$, we obtain a chain of corresponding open subfunctors of $\ell$-semistable sheaves
\[
    \Coh_{X,\omega,\gamma}^{GMss} \subset \Coh_{X,\omega,\gamma}^{\ell ss} \subset \Coh_{X,\omega,\gamma}^{ss} \subset \Coh_{X,\gamma}.
\]

The existence of moduli spaces of pure sheaves has been established in various situations:
\begin{itemize}
    \item The case of simple pure sheaves is settled in \cite{AltmanKleiman} and \cite{KosarewOkonek}.
    \item In the characteristic zero (AG) context, Simpson \cite{simpson1994moduli} proved the existence of a projective categorical moduli space for $\Coh_{X,\omega,\gamma}^{GMss}$ using techniques of Geometric Invariant Theory; the positive characteristic case is to be found in \cite{maruyama2016moduli}.
    \item Analogues of the Huybrechts-Lehn moduli spaces for $\ell$-semistable sheaves were constructed in \cite{pavel2022thesis}.
\end{itemize}

\subsection{Change of semistability and wall-crossing}
In this subsection we restrict the discussion to the algebraic geometric setting. When the dimension of $X$ is larger than one, the moduli spaces of semistable sheaves above depend on the choice of the polarization $\omega$ in the ample cone $\Amp^1(X)$. The situation which seems to appear in general is that for any numerical class $\gamma \in \Knum(X)$ there exists a locally finite set $\gW$ of real algebraic hypersurfaces $\gw \subset \Amp^1(X)$, called \textit{walls}, leading to a decomposition into connected components, called \textit{chambers}, of
\[
    \Amp^1(X) \setminus \bigcup_{\gw \in \gW} \gw
\]
accounting for the change of (semi)stability in the following sense. If $\omega_1, \omega_2$ are ample classes in the same chamber, then a coherent sheaf $E$ of class $\gamma$ is $\omega_1$-(semi)stable if and only if $E$ is $\omega_2$-(semi)stable. If this is indeed the case, then moduli spaces with respect to $\omega_1$ and $\omega_2$ coincide. The next step to understand the variation of the moduli spaces depending on the polarization would be to study wall-crossing, i.e. the relation between moduli spaces corresponding to adjacent chambers.

The existence of a chamber structure as above is guaranteed once the following stronger boundedness property for semistable sheaves is established. 

\begin{defn}[Uniform boundedness of semistable sheaves]
Given a numerical class $\gamma \in \Knum(X)$, we say that the \textit{uniform boundedness of semistability} holds for $\gamma$ if for any compact subset $K \subset \Amp^1(X)_\R$, the set of isomorphism classes of coherent sheaves of type $\gamma$ on $X$ that are $\omega$-semistable with respect to some $\omega \in K$ is bounded.
\end{defn}

In the surface case uniform boundedness is established using the Bogomolov inequality and the Hodge Index Theorem \cite{qin93equivalence,gottsche96change,FriedmanQin95,ellingsrud1995variation,matsuki97mumford}. In this case the resulting walls are moreover rational linear, a situation which is propitious to the study of wall-crossing phenomena, see the above references for such studies.

Another situation where a rational linear chamber structure exists is the case of two-dimensional pure sheaves \cite{pavel2024uniformboundednesssemistablepure}. 

In higher dimensions a wall and chamber structure exists \cite{greb2017compact}, however walls are neither linear nor rational in general \cite{qin93equivalence,schmitt00walls}. This problem was circumvented in \cite{greb16variation,greb16moduli} by introducing a more refined notion of semistability for which wall-crossing is well-behaved.



\bibliography{mainbib.bib}
\bibliographystyle{amsalpha}

\medskip
\medskip
\begin{center}
\rule{0.4\textwidth}{0.4pt}
\end{center}
\medskip
\medskip

\end{document}